\documentclass[a4paper,10pt,refcheck]{amsart}

\usepackage{acatta}

\newcommand{\blankSpace}{\cdot}
\newcommand{\mubar}{\bar{\mu}}
\newcommand{\Dol}{\operatorname{Dol}}

\newcommand{\canbund}{\omega}

\renewcommand{\Im}{\operatorname{Im}}
\renewcommand{\Re}{\operatorname{Re}}
\renewcommand*{\arraystretch}{1.5}

\title[Almost complex parallelizable manifolds]{Almost complex parallelizable manifolds: Kodaira dimension and special structures}

\author{Andrea Cattaneo}
\address{Universit\`a di Parma\\
Dipartimento di Scienze Matematiche, Fisiche e Informatiche\\
Unit\`a di Matematica e Informatica\\
Parco Area delle Scienze 53/A, 43124, Parma, Italy}
\email{andrea.cattaneo@unipr.it}

\author{Antonella Nannicini}
\address{Dipartimento di Matematica ed Informatica ``U. Dini''\\
Universit\`a degli Studi di Firenze\\
Viale Morgagni 67/A\\
50134 Firenze, Italy}
\email{antonella.nannicini@unifi.it}

\author{Adriano Tomassini}
\address{Universit\`a di Parma\\
Dipartimento di Scienze Matematiche, Fisiche e Informatiche\\
Unit\`a di Matematica e Informatica\\
Parco Area delle Scienze 53/A, 43124, Parma, Italy}
\email{adriano.tomassini@unipr.it}

\keywords{Kodaira dimension, almost complex manifolds, parallelizable manifolds}
\thanks{This work was partially supported by the Project PRIN 2017 ``Real and Complex Manifolds: Topology, Geometry and holomorphic dynamics'' and by GNSAGA of INdAM.}
\subjclass[2010]{32Q60, 53C56}

\begin{document}

\begin{abstract}
We study the Kodaira dimension of a real parallelizable manifold $M$, with an almost complex structure $J$ in standard form with respect to a given parallelism. For $X = (M, J)$ we give conditions under which $\kod(X) = 0$. We provide examples in the case $M = G \times G$, where $G$ is a compact connected real Lie group. Finally we describe geometrical properties of real parallelizable manifolds in the framework of statistical geometry.
\end{abstract}

\maketitle

\tableofcontents

\section{Introduction}

In complex geometry one of the most studied invariant of a manifold is its Kodaira dimension, which encodes the asymptotic behaviour of the dimension of the spaces of holomorphic pluricanonical sections on the given manifold. It can be defined in other different ways, e.g., as the maximal dimension of the image of the manifold via the pluricanonical maps or in the projective cases as the dimension of the canonical model of the original manifold (this works since the canonical ring of a smooth projective manifold is finitely generated by \cite{BCHM}). Recently the first definition of Kodaira dimension has been generalized in the context of almost complex geometry by Chen and Zhang (see \cite[Definition 1.2]{Chen-Zhang-1}) so we can speak of the \emph{Kodaira dimension of an almost complex manifold} (see also \cite{Cattaneo-Nannicini-Tomassini-1} and \cite{Cattaneo-Nannicini-Tomassini-2} for explicit computations of Kodaira dimension on some families of almost complex manifolds). Despite the fact that the definition proposed in \cite{Chen-Zhang-1} actually is the same for almost complex manifolds as it is for genuine complex manifolds, up to now it is not fully understood and also the techniques which can be used to compute it on concrete examples are still very few and basically rely on ad hoc procedures strictly linked with the geometry of the differentiable manifold underlying the almost complex manifold.

The class of parallelizable differentiable manifolds, i.e., those manifolds whose tangent bundle is trivial, is a particularly handy class of manifolds. For example, to define an almost complex structure on a parallelizable manifold it is enough to give a square matrix with entries in the smooth functions whose square is $-\id$. For this reason in the present paper we study some properties of almost complex manifolds whose underlying differentiable manifold is parallelizable. Observe that the fact that the \emph{real} tangent bundle of an almost complex manifold is trivial does not imply that also the \emph{pseudoholomorphic} tangent bundle is trivial as well, nor that the canonical bundle is pseudoholomorphically trivial. As a consequence, we study which conditions a smooth pluricanonical section must satisfy in order for it to be pseudoholomorphic, and we find that these conditions can be summed up into a single system of partial differential equations. We show that the solutions of this system must satisfy a certain elliptic equation of the second order, thus providing an analogue with the integrable case.

As an application of our results, we focus on a special type of almost complex structures which can be defined on the product of a real Lie group with itself. We show that for these manifolds the canonical bundle is pseudoholomorphically trivial if and only if the Lie group is unimodular. It is an interesting fact that our condition \eqref{eq: condition trivial canonical}, which expresses the unimodularity of the Lie group, is the same which appeared in \cite {Abbena-Grassi}.

In the last section we describe geometrical properties of almost complex real parallelizable manifolds in the framework of statistical geometry. Statistical structures play a central role in information geometry \cite {Amari}, \cite {Lauritzen}. This notion was extended to quasi-statistical structures, or statistical structures admitting torsion, by Kurose, \cite {K}. Starting from the construction of natural Norden structures, described in our previous paper \cite {Cattaneo-Nannicini-Tomassini-2} for almost complex $4-$dimensional solvmanifolds, and from a general construction of quasi-statistical structures, given in \cite {Blaga-Nannicini}, we define a natural family of quasi-statistical structures on a real parallelizable manifold $M$ of even dimension. Finally we observe that, by using the Sasaki metric, we can define a corresponding family of almost K\"ahler structures on the tangent bundle of $M$, $TM$. 

\begin{ack}
The authors wolud like to thank the anononymous referee for their review and the useful comments which allow us to improve the presentation of the paper.
\end{ack}

\section{Preliminaries}
\subsection{Preliminaries on almost complex manifolds}

\begin{defin}
Let $M$ be a smooth real manifold. An \emph{almost complex structure} on $M$ is an endomorphism $J: TM \longrightarrow TM$ such that $J^2 = -\id_{TM}$. An \emph{almost complex manifold} is a pair $X = (M, J)$ where $M$ is a smooth real manifold and $J$ is an almost complex structure on $M$.
\end{defin}

The existence of an almost complex structure on $M$ forces $\dim_{\IR} M$ to be even, say $\dim_{\IR} M = 2n$.

\begin{defin}
Let $X = (M, J)$ and $X' = (M', J')$ be almost complex manifolds. A \emph{pseudoholomorphic function} $f: X \longrightarrow X'$ is a smooth function $f: M \longrightarrow M'$ such that $df \circ J = J' \circ df$.
\end{defin}

We now show that the transition functions of $TM$ give an example of pseudoholomorphic functions.

\begin{prop}
Let $X = (M, J)$ be an almost complex manifold and let $\cU = \set{(U, \Phi_U)}$ be a trivializing open covering for the tangent bundle $TM$. Let $(U, \Phi_U)$, $(V, \Phi_V) \in \cU$ be such that $U \cap V \neq \varnothing$. Then the transition function
\[\Phi_{VU} = \Phi_V \circ \Phi_U^{-1}: \Phi_U(U \cap V) \longrightarrow \Phi_V(U \cap V)\]
is pseudoholomorphic.
\end{prop}
\begin{proof}
The open subsets $\Phi_U(U)$ and $\Phi_V(V)$ of $\IR^{2n}$ are almost complex manifolds by means of $(\Phi_U^{-1})^* J|_U$ and $(\Phi_V^{-1})^* J|_V$ respectively. The fact that $J$ is \emph{globally} defined on $M$ then implies that
\[d\Phi_{VU} \circ \pa{(\Phi_U^{-1})^* J|_U} = \pa{(\Phi_V^{-1})^* J|_V} \circ d\Phi_{VU}\]
i.e., that $\Phi_{VU}$ is pseudoholomorphic.
\end{proof}

\subsection{Dolbeault cohomology on almost complex manifolds}

Let $X = (M, J)$ be an almost complex manifold. From the complex point of view, the exterior differential $d$ splits as
\[d = \mu + \del + \delbar + \mubar, \qquad \text{with } \begin{array}{l}
\mu: \bigwedge^{p, q} X \longrightarrow \bigwedge^{p + 2, q - 1} X,\\
\del: \bigwedge^{p, q} X \longrightarrow \bigwedge^{p + 1, q} X,\\
\delbar: \bigwedge^{p, q} X \longrightarrow \bigwedge^{p, q + 1} X,\\
\mubar: \bigwedge^{p, q} X \longrightarrow \bigwedge^{p - 1, q + 2} X.
\end{array}\]
The condition that $d^2 = 0$ then translates into different equations involving these four operators, in particular we mention the following:
\[\mubar^2 = 0, \qquad \mubar \delbar + \delbar \mubar = 0, \qquad \del \mubar + \delbar^2 + \mubar \del = 0.\]

Observe that in the genuine almost complex setting it is not true that $\delbar^2 = 0$, so we can not define the Dolbeault cohomology of an almost complex manifold as in the integrable case. Following \cite{Cirici-Wilson} we define the generalized Dolbeault cohomology $H^{p, q}_{\Dol}(X)$ as follows. Since $\mubar^2 = 0$ we can define the $\mubar$-cohomology
\[H^{p, q}_{\mubar}(X) = \frac{\ker \pa{\Gamma\pa{X, \bigwedge^{p, q}X} \xrightarrow{\mubar} \Gamma\pa{X, \bigwedge^{p - 1, q + 2}X}}}{\im \pa{\Gamma\pa{X, \bigwedge^{p + 1, q - 2}X} \xrightarrow{\mubar} \Gamma\pa{X, \bigwedge^{p, q}X}}}\]
as usual. By the equation $\mubar \delbar + \delbar \mubar = 0$, the map $\delbar: H^{p, q}_{\mubar}(X) \longrightarrow H^{p, q + 1}_{\mubar}(X)$ defined on the class $[\alpha]_{\mubar} \in H^{p, q}_{\mubar}(X)$ represented by the $\mubar$-closed $(p, q)$-form $\alpha$ by $\delbar[\alpha]_{\mubar} = [\delbar \alpha]_{\mubar}$ is a well defined homomorphism. Finally, the equation $\del \mubar + \delbar^2 + \mubar \del = 0$ implies that $\delbar^2 = 0$ on $H^{p, q}_{\mubar}(X)$, so we define the \emph{Dolbeault cohomology} of $X$ as
\[H^{p, q}_{\Dol}(X) = \frac{\ker \pa{H^{p, q}_{\mubar}(X) \xrightarrow{\delbar} H^{p, q + 1}_{\mubar}(X)}}{\im \pa{H^{p, q - 1}_{\mubar}(X) \xrightarrow{\delbar} H^{p, q}_{\mubar}(X)}}.\]

\subsection{Kodaira dimension of an almost complex manifold}

In complex geometry one of the first invariants which are naturally attached to a manifold is the Kodaira dimension. This can equivalently be defined in two ways: given a \emph{complex} manifold $X$ its Kodaira dimension $\kod(X)$ is
\begin{enumerate}
\item the maximal dimension of the closure of the image of $X$ under the meromorphic maps associated with the pluricanonical systems $|\canbund_X^{\tensor m}|$,
\item the rate of growth of the spaces of holomorphic pluricanonical sections $H^0(X, \canbund_X^{\tensor m})$.
\end{enumerate}

Recently, Chen and Zhang in \cite{Chen-Zhang-1} showed that the last of these formulations can be successfully extended to the case where $X = (M, J)$ is an almost complex manifold. We give here a very brief account of \cite[$\S$3]{Chen-Zhang-1}. The usual delbar operator $\delbar$ naturally extends to a differential operator $\delbar: \canbund_X^{\tensor m} \longrightarrow \bigwedge^{0, 1} X \tensor \canbund_X^{\tensor m}$ and a smooth pluricanonical section $s \in \Gamma \pa{X, \canbund_X^{\tensor m}}$ is pseudoholomorphic if and only if $\delbar s = 0$. Thus we can still define
\[H^0(X, \canbund_X^{\tensor m}) = \set{s \in \Gamma \pa{X, \canbund_X^{\tensor m}} \st \delbar s = 0}.\]
Exploiting the theory of elliptic operators, in \cite[Theorem 3.6]{Chen-Zhang-1} it is shown that all these spaces are finite dimensional. Hence we have the following definition.

\begin{defin}[{cf.\ \cite[Definition 4.3]{Chen-Zhang-1}}]
Let $X = (M, J)$ be an almost complex manifold. The \emph{Kodaira dimension} of $X$ is the real number
\[\kod(X) = \left\{ \begin{array}{ll}
0 & \text{if } H^0\pa{X, \canbund_X^{\tensor m}} = 0 \quad \forall m \geq 1\\
{\displaystyle \limsup_{m \rightarrow +\infty} \frac{\log \dim H^0\pa{X, \canbund_X^{\tensor m}}}{\log m}} & \text{otherwise}.
\end{array} \right.\]
\end{defin}

\begin{rem}
It is possible to associate to an almost complex manifold its \emph{canonical ring}
\[R(X, \canbund_X) = \bigoplus_{m \geq 0} H^0(X, \canbund_X^{\tensor m}),\]
which is in a natural way a graded $\IC$-algebra (we have $H^0(X, \cO_X) \simeq \IC$ by Lemma \ref{lemma: liouville}). In the smooth projective case this algebra is known to be finitely generated (see \cite{BCHM} or \cite{Kawamata}) but almost nothing is known in the general setting of almost complex manifolds.
\end{rem}

\section{The Kodaira dimension of parallelizable manifolds}\label{parallelizable}

Let $X = (M, J)$ be an almost complex manifold and assume that $M$ is a real parallelizable manifold. Let $\set{v_1, w_1, \ldots, v_n, w_n}$ be a basis for $\Gamma\pa{M, TM}$ with the following properties:
\begin{enumerate}
\item it is a parallelism for $M$;
\item the almost complex structure $J$ is in standard form, namely
\[J v_i = w_i, \qquad J w_i = -v_i, \qquad i = 1, \ldots, n.\]
\end{enumerate}

From the complex point of view, we can introduce the vector fields of type $(1, 0)$
\[\cX_i = \frac{1}{2}\pa{v_i - \ii w_i}, \qquad i = 1, \ldots, n\]
and their dual $(1, 0)$-forms
\[\varphi^i = v^i + \ii w^i, \qquad i = 1, \ldots, n\]
where $\set{v^1, w^1, \ldots, v^n, w^n}$ is the dual basis of $\set{v_1, w_1, \ldots, v_n, w_n}$.

As a consequence
\[\psi = \varphi^1 \wedge \ldots \wedge \varphi^n\]
is a smooth section of $\canbund_X$, and analogously $\psi^{\tensor m}$ is a smooth section of $\canbund_X^{\tensor m}$. We want to compute $\delbar\pa{\psi^{\tensor m}}$ as a section of the bundle $\Lambda^{0, 1}_J M \tensor \canbund_X^{\tensor m}$.

\begin{lemma}
Let $X = (M, J)$ be an almost complex manifold, with $M$ real parallelizable. For every integer $m \geq 1$ we have
\[\delbar\pa{\psi^{\tensor m}} = m \alpha \tensor \psi^{\tensor m},\]
where $\alpha \in \Gamma\pa{X, \Omega^{0, 1}_X}$ is the smooth $(0, 1)$-form defined in \eqref{eq: alpha}.
\end{lemma}
\begin{proof}
For every $i = 1, \ldots, n$ there exist smooth functions $\lambda^i_{jk}$ on $X$ such that
\begin{equation}\label{eq: def lambda}
\delbar \varphi^i = \sum_{j, k = 1}^n \lambda^i_{jk} \varphi^j \wedge \bar{\varphi}^k.
\end{equation}
It follows that
\[\begin{array}{rl}
\delbar \psi = & \delbar \pa{\varphi^1 \wedge \ldots \wedge \varphi^n} =\\
= & \sum_{i = 1}^n (-1)^{i + 1} \varphi^1 \wedge \ldots \wedge \delbar \varphi^i \wedge \ldots \wedge \varphi^n =\\
= & \sum_{i = 1}^n (-1)^{i + 1} \varphi^1 \wedge \ldots \wedge \pa{\sum_{j, k = 1}^n \lambda^i_{jk} \varphi^j \wedge \bar{\varphi}^k} \wedge \ldots \wedge \varphi^n =\\
= & \sum_{i = 1}^n (-1)^{i + 1} \varphi^1 \wedge \ldots \wedge \pa{\sum_{k = 1}^n \lambda^i_{ik} \varphi^i \wedge \bar{\varphi}^k} \wedge \ldots \wedge \varphi^n =\\
= & \sum_{i, k = 1}^n -\lambda^i_{ik} \bar{\varphi}^k \wedge \varphi^1 \wedge \ldots \wedge \varphi^n =\\
= & \alpha \tensor \psi,
\end{array}\]
where
\begin{equation}\label{eq: alpha}
\alpha = - \sum_{k = 1}^n \alpha_k \bar{\varphi}^k, \qquad \alpha_k = \sum_{i = 1}^n \lambda^i_{ik}.
\end{equation}
Finally we observe that
\[\delbar\pa{\psi^{\tensor m}} = \sum_{i = 1}^m \psi \tensor \ldots \tensor \underbrace{\delbar \psi}_{i\text{-th place}} \tensor \ldots \tensor \psi = \sum_{i = 1}^m \alpha \tensor \psi^{\tensor m} = m \alpha \tensor \psi^{\tensor m},\]
which proves our result.
\end{proof}

\begin{rem}
It is easy to observe that
\[\alpha(\blank) = -\sum_{i = 1}^n \delbar \varphi^i(\cX_i, \blank).\]
\end{rem}

\begin{lemma}\label{lemma: pseudohol condition}
Let $X = (M, J)$ be an almost complex manifold, with $M$ real parallelizable. A smooth pluricanonical section $f \cdot \psi^{\tensor m}$ is pseudoholomorphic if and only if
\begin{equation}\label{eq: pseudohol cond 1}
\delbar f + mf \alpha = 0.
\end{equation}
\end{lemma}
\begin{proof}
By the Leibniz rule
\[\delbar\pa{f \cdot \psi^{\tensor m}} = \delbar f \otimes \psi^{\tensor m} + f \delbar\pa{\psi^{\tensor m}} = \pa{\delbar f + mf \alpha} \tensor \psi^{\tensor m}.\]
\end{proof}

\begin{cor}\label{cor: propagation of solutions}
Let $X = (M, J)$ be an almost complex manifold, with $M$ real parallelizable, let $f: M \longrightarrow \IC$ be a smooth function and let $t \in \IZ_{\geq 1}$. Then $f$ is a solution of $\delbar g + tg\alpha = 0$ if and only if $f^p$ is a solution of $\delbar g + ptg\alpha = 0$ for every $p \in \IZ_{\geq 1}$.
\end{cor}
\begin{proof}
The `only if' statement is clear, it suffices to choose $p = 1$. For the `if' statement observe that
\[\delbar(f^p) + pt f^p \alpha = p f^{p - 1} \delbar f + pt f^p \alpha = p f^{p - 1} \pa{\delbar f + tf \alpha} = 0.\]
\end{proof}

Focusing on the smooth canonical section $\psi$ we have two cases: either it is pseudoholomorphic or not. We show that, as one may expect, in the former case the Kodaira dimension of $X$ is $0$.

We now show that one of the most known and used results for complex manifolds, Liouville's theorem, still holds in the context of almost complex geometry. Our proof relies on the argument in \cite[Theorem 3.6]{Chen-Zhang-1}, which is the essential ingredient to use the maximum principle.

\begin{lemma}[Liouville's Theorem]\label{lemma: liouville}
Let $X = (M, J)$ be an almost complex manifold, with $M$ compact. Then $H^0(X, \cO_X) \simeq \IC$.
\end{lemma}
\begin{proof}
Let $h$ be any Hermitian metric on $X$ and denote by $*$ its associate Hodge operator. Let $f: X \longrightarrow \IC$ be any smooth complex valued function on $X$, then we trivially have that $\del * f = 0$ since $*f$ is an $(n, n)$-form. It follows that $\delbar f = 0$ if and only if $\Delta_{\delbar} f = 0$, where $\Delta_{\delbar}$ is the Dolbeault Laplacian $\Delta_{\delbar} = \delbar^* \delbar + \delbar \delbar^*$, i.e., pseudoholomorphic and harmonic functions coincide. Now, by \cite[Theorem 3.6]{Chen-Zhang-1} the equation $\Delta_{\delbar} f = 0$ is an elliptic equation and by the compactness of $X$ and the maximum principle we deduce that $f$ must be constant.
\end{proof}

\begin{prop}\label{prop: kod = 0}
Let $X = (M, J)$ be an almost complex manifold, with $M$ real parallelizable. If $\delbar \psi = 0$, then $\kod(X) = 0$.
\end{prop}
\begin{proof}
Under the assumption that $\delbar \psi = 0$, a smooth pluricanonical section $f \cdot \psi^{\tensor m}$ is pseudoholomorphic if and only if $\delbar f = 0$. Hence
\[\begin{array}{ccc}
H^0\pa{X, \cO_X} & \longrightarrow & H^0\pa{X, \canbund_X^{\otimes m}}\\
f & \longmapsto & f \cdot \psi^{\otimes m}
\end{array}\]
is an isomorphism for every integer $m \geq 1$. By Liouville's Theorem we have then that $P_m(X) = 1$ for all $m \geq 1$, which readily implies the proposition.
\end{proof}

\begin{rem}
The converse of Proposition \ref{prop: kod = 0} does not hold in general, and a counterexample can be found in \cite[$\S$5.1]{Cattaneo-Nannicini-Tomassini-2}. There, the authors consider a family of almost complex structures $J_a$ ($a \in \IR \smallsetminus \set{0}$) on a $4$-dimensional nilmanifold $\cN$. In this example it is shown that (with our notation)
\[\alpha_a = \ii \frac{a}{4} \bar{\varphi}^2_a,\]
so $\alpha_a \neq 0$ for all $a \in \IR \smallsetminus \set{0}$, but (see \cite[Proposition 5.5]{Cattaneo-Nannicini-Tomassini-2})
\[\kod(\cN, J_a) = \left\{ \begin{array}{ll}
-\infty & \text{if } a \notin 2\pi \IQ,\\
0 & \text{if } a \in 2\pi \IQ \smallsetminus \set{0}.
\end{array} \right.\]
\end{rem}

\begin{rem}
Observe that $\psi$ is a nowhere vanishing canonical section, hence the condition $\delbar \psi = 0$ forces the canonical bundle to be pseudoholomorphically trivial, i.e., there exists a pseudoholomorphic isomorphism $\canbund_X \simeq \cO_X$. It may happen that $\canbund_X$ is not trivial, but $\canbund_X^{\tensor m} \simeq \cO_X$ for some $m \in \IZ_{\geq 2}$, e.g., this is the case of bielliptic surfaces. In this case $\alpha \neq 0$ for every real parallelism $\set{v_1, w_1, \ldots, v_n, w_n}$ as we are considering, but $\kod(X) = 0$.
\end{rem}

\subsection{The class of \texorpdfstring{$\alpha$}{alpha} in Dolbeault cohomology}

We want now to see how the $(0, 1)$-form $\alpha$ changes if one changes the parallelism.

\begin{lemma}
Let $\cP = \set{v_1, w_1, \ldots, v_n, w_n}$ and $\cP' = \set{v'_1, w'_1, \ldots, v'_n, w'_n}$ be two real parallelisms on $M$ in which the almost complex structure $J$ is in its standard form, and denote by $\set{\varphi^1, \ldots, \varphi^n}$ and $\set{{\varphi'}^1, \ldots, {\varphi'}^n}$ the corresponding coframes of $(1, 0)$-forms. Let $\alpha$ and $\alpha'$ be the $(0, 1)$-forms associated to $\cP$ and $\cP'$ respectively, as in \eqref{eq: alpha}. Then there exists a nowhere vanishing complex valued smooth function $g: X \longrightarrow \IC$ such that
\[\alpha' =  \alpha + \frac{\delbar g}{g}.\]
\end{lemma}
\begin{proof}
Since $M$ is parallelizable, we can write
\[{\varphi'}^j = \sum_{i = 1}^n m_{ij} \varphi^i\]
for suitable complex valued functions $m_{ij} = m_{ij}(x)$ on $M$ such that $\cM = \pa{m_{ij}}$ is a complex valued invertible matrix for all $x \in M$. It follows that
\[\psi' = {\varphi'}^1 \wedge \ldots \wedge {\varphi'}^n = \det(\cM) \cdot \varphi^1 \wedge \ldots \wedge \varphi^n = \det(\cM) \cdot \psi,\]
and so
\[\begin{array}{rl}
\delbar \psi' = & \delbar \det(\cM) \wedge \psi + \det(\cM) \delbar \psi =\\
= & \frac{\delbar \det(\cM)}{\det(\cM)} \wedge \det(\cM) \psi + \alpha \wedge \det(\cM) \psi =\\
= & \pa{\frac{\delbar \det(\cM)}{\det(\cM)} + \alpha} \wedge \psi'.
\end{array}\]
On the other hand $\delbar \psi' = \alpha' \wedge \psi'$ and so we deduce that
\[\alpha' =  \alpha + \frac{\delbar \det(\cM)}{\det(\cM)}.\]
\end{proof}

\begin{lemma}
Let $X = (M, J)$ be an almost complex manifold with $M$ real parallelizable. With the notations introduced before, we have that
\[\delbar^2 \psi = \delbar \alpha \wedge \psi.\]
In particular, if $J$ is integrable then $\delbar \alpha = 0$.
\end{lemma}
\begin{proof}
Since $\delbar \psi = \alpha \wedge \psi$ we easily deduce that
\[\delbar^2 \psi = \delbar \pa{\alpha \wedge \psi} = \delbar \alpha \wedge \psi - \alpha \wedge \delbar \psi = \delbar \alpha \wedge \psi - \alpha \wedge (\alpha \wedge \psi) = \delbar \alpha \wedge \psi.\]
If $J$ is integrable we have $\delbar^2 \psi = 0$, so $\delbar \alpha \wedge \psi = 0$ from which $\delbar \alpha = 0$.
\end{proof}

From the previous lemma, it follows that in the integrable case $\delbar \alpha$ defines the class $[\delbar \alpha] = 0 \in H^{0, 2}_{\delbar}(X)$. We want to show that this is still true even in the non-integrable case.

\begin{prop}
Let $X = (M, J)$ be an almost complex manifold with $M$ real parallelizable. With the same notations as above, we have
\[[\delbar \alpha]_{\Dol} = 0 \in H^{0, 2}_{\Dol}(X).\]
\end{prop}
\begin{proof}
Observe that since $\alpha$ is a $(0, 1)$-form, then $\mubar \alpha = 0$ for bidegree reasons and so the class $[\alpha]_{\mubar}$ is well defined in $H^{0, 1}_{\mubar}(X)$. Observe that since $\delbar \mubar + \mubar \delbar = 0$ we have also that $\mubar \delbar \alpha = -\delbar \mubar \alpha = 0$ and so also $[\delbar \alpha]_{\mubar} \in H^{0, 2}_{\mubar}(X)$ is well defined. But then
\[[\delbar \alpha]_{\mubar} = \delbar [\alpha]_{\mubar}\]
and so $[\delbar \alpha]_{\Dol} = 0 \in H^{0, 2}_{\Dol}(X)$.
\end{proof}

\section{A PDE system}

As we saw in Lemma \ref{lemma: pseudohol condition}, the condition that $\delbar\pa{f \psi^{\tensor m}} = 0$ is equivalent to
\[\delbar f + mf \alpha = 0.\]
This condition (which is \eqref{eq: pseudohol cond 1}) is equivalent to the system of equations
\begin{equation}\label{eq: pseudohol cond 2}
\bar{\cX}_k(f) - m \alpha_k f = 0, \qquad k = 1, \ldots, n,
\end{equation}
and we want to write this last condition from the real point of view. To this end we write
\[f = u + \ii v,\]
and so we see that \eqref{eq: pseudohol cond 2} is equivalent to the system
\begin{equation}\label{eq: pseudohol cond 3}
\begin{array}{l}
v_k(u) - w_k(v) = 2m \pa{\Re\pa{\alpha_k} u - \Im\pa{\alpha_k} v}\\
v_k(v) + w_k(u) = 2m \pa{\Re\pa{\alpha_k} v + \Im\pa{\alpha_k} u}
\end{array}, \qquad k = 1, \ldots, n.
\end{equation}

\begin{rem}\label{rem: symmetry of solutions}
Observe that if $(u, v)$ is a solution of \eqref{eq: pseudohol cond 3}, then $(-v, u)$ is also a solution. From the complex point of view, this corresponds to the fact that the space of solutions of \eqref{eq: pseudohol cond 3} is closed under multiplication by $\ii$ (i.e., if $f$ is a solution, then $\ii \cdot f$ is also a solution).
\end{rem}

From \eqref{eq: pseudohol cond 3} one sees that
{\renewcommand*{\arraystretch}{1.5}
\[\begin{array}{rl}
v_k(w_k(u)) = & -v_k(v_k(v)) + 2m \Re\pa{\alpha_k} v_k(v) + 2m \Im\pa{\alpha_k} w_k(v) +\\
 & + 2m \pa{v_k\pa{\Re\pa{\alpha_k}} - 2m \pa{\Im\pa{\alpha_k}}^2} v +\\
 & + 2m \pa{v_k\pa{\Im\pa{\alpha_k}} + 2m \Re\pa{\alpha_k} \Im\pa{\alpha_k}} u,\\
w_k(v_k(u)) = & w_k(w_k(v)) - 2m \Re\pa{\alpha_k} v_k(v) - 2m \Im\pa{\alpha_k} w_k(v) +\\
 & + 2m \pa{2m \pa{\Re\pa{\alpha_k}}^2 - w_k\pa{\Im\pa{\alpha_k}}} v +\\
 & + 2m \pa{w_k\pa{\Re\pa{\alpha_k}} + 2m \Re\pa{\alpha_k} \Im\pa{\alpha_k}} u,
\end{array}\]}
and so
\begin{equation}\label{eq: commutator 1}
{\renewcommand*{\arraystretch}{1.5}
\begin{array}{rl}
[w_k, v_k](u) = & v_k(v_k(v)) + w_k(w_k(v)) +\\
 & - 4m \Re\pa{\alpha_k} v_k(v) - 4m \Im\pa{\alpha_k} w_k(v) +\\
 & - 2m \pa{v_k\pa{\Re\pa{\alpha_k}} + w_k\pa{\Im\pa{\alpha_k}} - 2m \abs{\alpha_k}^2} v +\\
 & - 2m \pa{v_k\pa{\Im\pa{\alpha_k}} - w_k\pa{\Re\pa{\alpha_k}}} u.
\end{array}}
\end{equation}

\begin{rem}
Observe that
\[v_k\pa{\Re\pa{\alpha_k}} + w_k\pa{\Im\pa{\alpha_k}} = 2 \cdot \Re\pa{\cX_k \pa{\alpha_k}}, \, v_k\pa{\Im\pa{\alpha_k}} - w_k\pa{\Re\pa{\alpha_k}} = 2 \cdot \Im\pa{\cX_k \pa{\alpha_k}},\]
hence the coefficient of $u$ in \eqref{eq: commutator 1} can be written as
\begin{equation}\label{eq: coeff 1}
-4m \cdot \Im\pa{\cX_k \pa{\alpha_k}}.
\end{equation}
\end{rem}

Observe that we can express
\[[w_k, v_k] = 2 \ii [\cX_k, \bar{\cX}_k] = 2 \sum_{i = 1}^n \Im(\lambda^i_{kk}) v_i - \Re(\lambda^i_{kk}) w_i,\]
and so, using \eqref{eq: pseudohol cond 3}, we obtain that
\begin{equation}\label{eq: commutator 2}
\begin{array}{rl}
[w_k, v_k](u) = & 2 \sum_{i = 1}^n (\Im(\lambda^i_{k, k}) \pa{w_i(v) + 2m \Re(\alpha_i)u - 2m \Im(\alpha_i)v} +\\
 & - \Re(\lambda^i_{kk}) \pa{-v_i(v) + 2m \Re(\alpha_i)v + 2m \Im(\alpha_i)u}) =\\
= & \sum_{i = 1}^n (2\Re(\lambda^i_{kk}) v_i(v) + 2\Im(\lambda^i_{kk}) w_i(v) +\\
 & + 4m \pa{\Im(\lambda^i_{kk}) \Re(\alpha_i) - \Re(\lambda^i_{kk}) \Im(\alpha_i)} u +\\
 & - 4m \pa{\Im(\lambda^i_{kk}) \Im(\alpha_i) + \Re(\lambda^i_{kk}) \Re(\alpha_i)} v) =\\
= & \sum_{i = 1}^n (2\Re(\lambda^i_{kk}) v_i(v) + 2\Im(\lambda^i_{kk}) w_i(v) +\\
 & + 4m \Im(\lambda^i_{kk} \bar{\alpha}_i) u  - 4m \Re(\lambda^i_{kk} \bar{\alpha}_i) v).
\end{array}
\end{equation}

Comparing \eqref{eq: commutator 1} and \eqref{eq: commutator 2} we have then the relation
\begin{equation}\label{eq: main rel}
\begin{array}{l}
v_k(v_k(v)) + w_k(w_k(v)) +\\
-4m\Re(\alpha_k) v_k(v) - 4m\Im(\alpha_k)w_k(v) +\\
- 2\sum_{i = 1}^n \Re(\lambda^i_{kk}) v_i(v) - 2 \sum_{i = 1}^n \Im(\lambda^i_{kk}) w_i(v) +\\
- 4m \pa{\Re\pa{\cX_k(\alpha_k)} - m|\alpha_k|^2 - \sum_{i = 1}^n \Re(\lambda^i_{kk} \bar{\alpha}_i)} v +\\
- 4m \pa{\Im\pa{\cX_k(\alpha_k)} + \sum_{i = 1}^n \Im(\lambda^i_{kk} \bar{\alpha}_i)} u = 0,
\end{array}
\end{equation}
and, analogously, the `twin' relation
\[\begin{array}{l}
-v_k(v_k(u)) - w_k(w_k(u)) +\\
+4m\Re(\alpha_k) v_k(u) + 4m\Im(\alpha_k)w_k(u) +\\
+ 2\sum_{i = 1}^n \Re(\lambda^i_{kk}) v_i(u) + 2 \sum_{i = 1}^n \Im(\lambda^i_{kk}) w_i(u) +\\
+ 4m \pa{\Re\pa{\cX_k(\alpha_k)} - m|\alpha_k|^2 - \sum_{i = 1}^n \Re(\lambda^i_{kk} \bar{\alpha}_i)} u +\\
- 4m \pa{\Im\pa{\cX_k(\alpha_k)} + \sum_{i = 1}^n \Im(\lambda^i_{kk} \bar{\alpha}_i)} v = 0.
\end{array}\]

Take now the sum over $k$ of \eqref{eq: main rel} and the twin relation. Focusing on the first sum, the coefficient of $u$ is
\begin{equation}\label{eq: coeff u}
-4m \Im\pa{\sum_{k = 1}^n \pa{\cX_k(\alpha_k) + \sum_{i = 1}^n \lambda^i_{kk} \bar{\alpha}_i}}
\end{equation}
while the coefficient of $v$ is
\begin{equation}\label{eq: coeff v}
-4m \Re \pa{\sum_{k = 1}^n \pa{\cX_k(\alpha_k) - m |\alpha_k|^2 - \sum_{i = 1}^n \lambda^i_{kk} \bar{\alpha}_i}}.
\end{equation}

\begin{rem}
If we consider the sum of the twin relations, the coefficient of $v$ is the opposite of \eqref{eq: coeff u} and the coefficient of $u$ equals \eqref{eq: coeff v}.
\end{rem}

To simplify the notation we set
\begin{equation}\label{eq: beta}
\beta = \sum_{k = 1}^n \pa{\cX_k(\alpha_k) + \sum_{i = 1}^n \lambda^i_{kk} \bar{\alpha}_i},
\end{equation}
then $\beta$ is a smooth complex valued function on our manifold $M$. Moreover, we set
\[\begin{array}{rl}
L_m(\blank) = & \sum_{k = 1}^n \left( v_k(v_k(\blank)) + w_k(w_k(\blank)) + \right.\\
 & -4m\Re(\alpha_k) v_k(\blank) - 4m\Im(\alpha_k)w_k(\blank) +\\
 & \left. - 2\sum_{i = 1}^n \Re(\lambda^i_{kk}) v_i(\blank) - 2 \sum_{i = 1}^n \Im(\lambda^i_{kk}) w_i(\blank) \right).
\end{array}\]

\begin{rem}
The operator $L_m$ is a second order elliptic operator on $M$.
\end{rem}

With these notations, we want to solve the system
\begin{equation}\label{eq: final system}
\left\{ \begin{array}{l}
L_m(v) - 4m \Re \pa{\sum_{k = 1}^n \pa{\cX_k(\alpha_k) - m |\alpha_k|^2 - \sum_{i = 1}^n \lambda^i_{kk} \bar{\alpha}_i}} v - 4m \Im(\beta) u = 0\\
-L_m(u) + 4m \Re \pa{\sum_{k = 1}^n \pa{\cX_k(\alpha_k) - m |\alpha_k|^2 - \sum_{i = 1}^n \lambda^i_{kk} \bar{\alpha}_i}} u - 4m \Im(\beta) v = 0.
\end{array} \right.
\end{equation}

\begin{prop}\label{prop: system decouples}
Let $\beta$ be as in \eqref{eq: beta}. If $\beta$ is a \emph{real} valued function, then system \eqref{eq: final system} decouples, and $u$ and $v$ satisfy the same second order elliptic equation.
\end{prop}
\begin{proof}
This is obvious once we impose the condition $\Im(\beta) = 0$ in \eqref{eq: final system}.
\end{proof}

\begin{thm}\label{thm: constant solution}
Let $\beta$ be as in \eqref{eq: beta}. If $\beta$ is a \emph{real} valued function and
\[\Re \pa{\sum_{k = 1}^n \pa{\cX_k(\alpha_k) - \sum_{i = 1}^n \lambda^i_{kk} \bar{\alpha}_i}} \geq m \sum_{k = 1}^n |\alpha_k|^2,\]
then any solution of \eqref{eq: final system} is constant.
\end{thm}
\begin{proof}
By Proposition \ref{prop: system decouples} the system \eqref{eq: final system} decouples and $u$, $v$ satisfy the same second order elliptic equation. As a consequence we only need to prove that any solution $v$ of
\[L_m(v) - 4m \Re \pa{\sum_{k = 1}^n \pa{\cX_k(\alpha_k) - m |\alpha_k|^2 - \sum_{i = 1}^n \lambda^i_{kk} \bar{\alpha}_i}} v = 0\]
is constant. Since the manifold $M$ is compact, any solution $v$ attains its maximum $C$ and minimum $c$. Moreover, $-v$ is also a solution with maximun $-c$ and minimum $-C$. If $C \geq 0$, then $v$ must be constant by the maximum principle and our hypothesis. If $C < 0$, then we have $-c \geq -C > 0$ and so by the maximum principle again on $-v$ we deduce that $-v$, hence $v$, is constant.
\end{proof}

\begin{cor}\label{cor: h^0}
With the same hypothesis as Theorem \eqref{thm: constant solution} we have that
\[H^0(X, \canbund_X^{\tensor m}) = \left\{ \begin{array}{ll}
0 & \text{if } \alpha \neq 0,\\
\IC & \text{if } \alpha = 0.
\end{array} \right.\]
\end{cor}
\begin{proof}
We know by Theorem \ref{thm: constant solution} that the only solutions to \eqref{eq: final system}, hence to \eqref{eq: pseudohol cond 3}, are constants. Moreover \eqref{eq: pseudohol cond 3} becomes
\[\begin{array}{l}
\Re\pa{\alpha_k} u - \Im\pa{\alpha_k} v = 0\\
\Re\pa{\alpha_k} v + \Im\pa{\alpha_k} u = 0
\end{array}, \qquad k = 1, \ldots, n.
\]
Assume that $\alpha \neq 0$. This means that there is at least a coefficient $\alpha_k$ which is non-zero at some point $x \in M$. By easy linear algebra we have
\[\left\{ \begin{array}{l}
\Re\pa{\alpha_k(x)} u - \Im\pa{\alpha_k(x)} v = 0\\
\Re\pa{\alpha_k(x)} v + \Im\pa{\alpha_k(x)} u = 0
\end{array} \right. \Longrightarrow u = v = 0,
\]
and so $H^0(X, \canbund_X^{\tensor m}) = 0$. If $\alpha = 0$, then we already observed in Proposition \ref{prop: kod = 0} that $H^0(X, \canbund_X^{\tensor m}) = \IC$.
\end{proof}

\subsection{Parallelizable manifolds which are quotients of Lie groups}

In this section we want to investigate more in detail the case where the underlying differentiable manifold $M$ is a quotient of a real Lie group by the action of a discrete cocompact subgroup, i.e., by a lattice.

Assume that the manifold $M$ is the quotient of a (real) Lie group $G$ by a lattice, and that the fields $v_1, w_1, \ldots, v_n, w_n$ are obtained as left invariant vector fields on $G$. By \cite[Proposition 2.5]{Wolf1} all the $\lambda^i_{jk}$'s appearing in \eqref{eq: def lambda} are constant, hence expression \eqref{eq: beta} simplifies to
\[\beta = \sum_{i, k = 1}^n \lambda^i_{kk} \bar{\alpha}_i,\]
and system \eqref{eq: final system} takes the easier form
\begin{equation}\label{eq: final system easy}
\left\{ \begin{array}{l}
L_m(v) + \pa{4m^2 \sum_{k = 1}^n |\alpha_k|^2 + 4m \Re(\beta)} v - 4m \Im(\beta) u = 0\\
-L_m(u) - \pa{4m^2 \sum_{k = 1}^n |\alpha_k|^2 + 4m \Re(\beta)} u - 4m \Im(\beta) v = 0.
\end{array} \right.
\end{equation}

Observe that in this case $\beta$ is a \emph{constant} function, i.e., we can assume that $\beta \in \IC$. The general results obtained in the previous section can then be simplified as follows.

\begin{prop}[cf.\ Proposition \ref{prop: system decouples}]
If $\beta \in \IR$, then system \eqref{eq: final system easy} decouples, and $u$ and $v$ satisfy the same second order elliptic equation.
\end{prop}

\begin{thm}[cf.\ Theorem \ref{thm: constant solution}]
If $\beta \in \IR$ and
\[\beta + m \sum_{k = 1}^n |\alpha_k|^2 \leq 0,\]
then any solution of \eqref{eq: final system easy} is constant.
\end{thm}

\subsection{Some remarks}

Let us return to the general setting, where $X = (M, J)$ is an almost complex manifold with $M$ real parallelizable. Assume that the function $\beta$ of \eqref{eq: beta} is real valued, so that the system \eqref{eq: final system} decouples. We want to focus on the extra condition of Theorem \ref{thm: constant solution}, namely that
\begin{equation}\label{eq: condition}
\Re \pa{\sum_{k = 1}^n \pa{\cX_k(\alpha_k) - \sum_{i = 1}^n \lambda^i_{kk} \bar{\alpha}_i}} \geq m \sum_{k = 1}^n |\alpha_k|^2.
\end{equation}

If $\alpha = 0$ then we showed in Proposition \ref{prop: kod = 0} that $H^0(X, \canbund_X^{\tensor m}) \simeq \IC$ for all $m \geq 1$ \emph{independently} of the fact that this extra condition is satisfied or not. On the contrary, if $\alpha \neq 0$ such condition is essential to conclude that $H^0(X, \canbund_X^{\tensor m}) = 0$ in Corollary \ref{cor: h^0}.

We now want to discuss a bit on this condition. Assume that $\alpha \neq 0$: then $\alpha_k \neq 0$ for some $k$ and so
\[m \sum_{k = 1}^n |\alpha_k|^2 > 0 \qquad \Longrightarrow \qquad m \sum_{k = 1}^n |\alpha_k|^2 \xrightarrow{m \rightarrow +\infty} +\infty.\]
As a consequence, we can use condition \eqref{eq: condition} at most only for a finite number of integers $m$, and what we conclude is that the corresponding pluricanonical bundles $\canbund_X^{\tensor m}$ have no pseudoholomorphic sections. By Corollary \ref{cor: propagation of solutions} the only information we can get is that for higher tensor powers of $\canbund_X$ (the one corresponding to multiples of the aforementioned values of $m$) the only pseudoholomprhic section we can detect is the trivial one, and this actually gives us no information about Kodaira dimension. 

\section{Examples}

Let $G$ be a connected real Lie group of real dimension $n$, not necessarily compact. Fix a global coframe $\set{e^1, \ldots, e^n}$ on $T^* G$, which is a parallelization and consider the structure equations with respect to it:
\begin{equation}\label{eq: real structure}
d e^i = \sum_{1 \leq j < k \leq n} \mu^i_{jk} e^j \wedge e^k, \qquad i = 1, \ldots, n.
\end{equation}
Let now $M = G \times G$ and denote $\pi_r: M \longrightarrow G$ the projection to the $r$-th factor, for $r = 1, 2$. Then $M$ is a $2n$-dimensional compact manifold, which is parallelized by the following coframe for $T^* M$:
\[v^i = \pi_1^* e^i, \qquad w^i = \pi_2^* e^i, \qquad i = 1, \ldots, n.\]

We define on $M$ the almost complex structure $J$ by requiring that the bundle of $(1, 0)$-forms is generated by
\[\varphi^i = v^i + \ii w^i.\]
Observe that with this definition the set of dual vector fields $\set{v_1, w_1, \ldots, v_n, w_n}$ is a parallelization of $M$ with respect to which the almost complex structure $J$ is in standard form, hence we are in the framework introduced in Section \ref{parallelizable}.

The following lemma describes $d\varphi^i$: it is a straightforward computation, so we omit its proof.

\begin{lemma}
We have that
\begin{equation}
\begin{array}{rl}
\del \varphi^i = & \frac{1}{4}\pa{1 - \ii} \sum_{1 \leq j < k \leq n} \mu^i_{jk} \varphi^j \wedge \varphi^k;\\
\delbar \varphi^i = & \frac{1}{4} \pa{1 + \ii} \sum_{1 \leq j < k \leq n} \mu^i_{jk} \pa{\varphi^j \wedge \bar{\varphi}^k - \varphi^k \wedge \bar{\varphi}^j};\\
\mubar \varphi^i = & \frac{1}{4}\pa{1 - \ii} \sum_{1 \leq j < k \leq n} \mu^i_{jk} \bar{\varphi}^j \wedge \bar{\varphi}^k.
\end{array}
\end{equation}
\end{lemma}

\begin{cor}
The almost complex structure $J$ is integrable if and only if $de^i = 0$ for all $i = 1, \ldots, n$.
\end{cor}
\begin{proof}
In fact, $J$ is integrable if and only if $\mubar \equiv 0$, which is equivalent to $\mu^i_{jk} = 0$ for every $i$, $j$ and $k$. In turn this is equivalent to $de^i = 0$ for all $i = 1, \ldots, n$.
\end{proof}

\begin{cor}
We have that
\[\lambda^i_{jk} = \left\{ \begin{array}{ll}
\frac{1}{4} \pa{1 + \ii} \mu^i_{jk} & \text{if } j < k,\\
0 & \text{if } j = k,\\
-\frac{1}{4} \pa{1 + \ii} \mu^i_{kj} & \text{if } j > k.
\end{array}\right.\]
\end{cor}

Using \eqref{eq: alpha} it is now possible to compute that
\begin{equation}
\begin{array}{rl}
\alpha_k = & \sum_{i = i}^n \lambda^i_{ik} =\\
= & \sum_{1 \leq i < k \leq n} \frac{1}{4} \pa{1 + \ii} \mu^i_{ik} + \sum_{1 \leq k < i \leq n} \pa{-\frac{1}{4}} \pa{1 + \ii} \mu^i_{ki} =\\
= & \frac{1}{4} \pa{1 + \ii} \pa{\sum_{1 \leq i < k \leq n} \mu^i_{ik} - \sum_{1 \leq k < i \leq n} \mu^i_{ki}}.
\end{array}
\end{equation}
As a consequence we have the following proposition.

\begin{prop}\label{prop: trivial canonical}
With the same notations used in this section, we have that
\begin{equation}\label{eq: condition trivial canonical}
\delbar \pa{\varphi^1 \wedge \ldots \wedge \varphi^n} = 0 \Longleftrightarrow \sum_{1 \leq i < k \leq n} \mu^i_{ik} = \sum_{1 \leq k < i \leq n} \mu^i_{ki} \text{ for all } k = 1, \ldots, n.
\end{equation}
\end{prop}

Let $\gothg$ be the Lie algebra of $G$, we can then see the global frame $\set{e_1, \ldots, e_n}$ dual to $\set{e^1, \ldots, e^n}$ as a basis for $\gothg$. The adjoint representation is
\[\begin{array}{rccc}
\operatorname{ad}: & \gothg & \longrightarrow & \End(\gothg)\\
 & X & \longmapsto &  [X, \blank].
\end{array}\]

\begin{prop}\label{prop: unimod cond}
With the notations introduced in this section, we have that $\operatorname{tr}\pa{\operatorname{ad}(X)} = 0$ for every $X \in \gothg$ if and only if
\[\sum_{1 \leq i < k} \mu^i_{ik} = \sum_{k < i \leq n} \mu^i_{ki} \qquad \text{for all } k = 1, \ldots, n.\]
\end{prop}
\begin{proof}
We work explicitly with the basis $\set{e_1, \ldots, e_n}$ for $\gothg$. Since taking the trace is a linear operation we have that $\operatorname{tr}\pa{\operatorname{ad}(X)} = 0$ for every $X \in \gothg$ if and only if $\operatorname{tr}\pa{\operatorname{ad}(e_i)} = 0$ for every $i = 1, \ldots, n$. Fix then $i \in \set{1, \ldots, n}$, and observe that
\[\operatorname{ad}(e_i)(e_j) = [e_i, e_j] = \left\{ \begin{array}{ll}
\sum_{k = 1}^n (-\mu^k_{ij}) e_k & \text{if } i < j,\\
0 & \text{if } i = j,\\
\sum_{k = 1}^n (\mu^k_{ji}) e_k & \text{if } i > j.
\end{array} \right.\]
This means that the matrix representing $\operatorname{ad}(e_i)$ in the given basis is
\[M_i = \pa{m^i_{jk}} \qquad \text{with} \qquad m^i_{jk} = \left\{ \begin{array}{ll}
-\mu^k_{ij} & \text{if } i < j,\\
0 & \text{if } i = j,\\
\mu^k_{ji} & \text{if } i > j.
\end{array} \right.\]
As a consequence
\[\operatorname{tr}\pa{\operatorname{ad}(e_i)} = \sum_{j = 1}^n m^i_{jj} = \sum_{1 \leq j < i} \mu^j_{ji} - \sum_{i< j \leq n} \mu^j_{ij}\]
which readily implies the proposition.
\end{proof}

Proposition \ref{prop: unimod cond} readily implies the following result.

\begin{thm}\label{thm: unimodular has kod dim 0}
Let $G$ be a connected real Lie group, and let $X = (G \times G, J)$ where $J$ is the left invariant almost complex structure defined in this section. Then $G$ is unimodular if and only if $\delbar\pa{\varphi^1 \wedge \ldots \wedge \varphi^n} = 0$. In particular, in this case the canonical bundle of $X$ is pseudoholomorphically trivial and $\kod(X) = 0$.
\end{thm}

\begin{rem}
Since on a Lie algebra we have that for $1 \leq j < k \leq n$
\[[e_j, e_k] = \sum_{i = 1}^n e^i\pa{[e_j, e_k]} e_i = \sum_{i = 1}^n -de^i(e_j, e_k) e_i = -\sum_{i = 1}^n \mu^i_{jk} e_i\]
we can read the coefficients $\mu^i_{jk}$ directly from the expression of $[e_j, e_k]$.
\end{rem}

By \cite[Lemma 6.2]{Milnor}, any connected Lie group which admits a compact quotient (hence, in particular, any compact Lie group) is unimodular. We now present some examples constructed from compact Lie groups, as well as an example from a non-unimodular Lie group. In the exposition we will write the equations for the Lie bracket on a basis for $\gothg$ and show directly that the condition of Proposition \ref{prop: trivial canonical} is fulfilled, without making use of Theorem \ref{thm: unimodular has kod dim 0} (this is possible since our examples have small dimension).

\begin{exam}[$G = \SU(2)$]\label{exam: spheres}
In this example, let $G = \SU(2)$. Its Lie algebra is
\[\mathfrak{su}(2) = \left\langle \underbrace{\left( \begin{array}{cc}
\ii & 0\\
0 & -\ii
\end{array} \right)}_{e_1}, \underbrace{\left( \begin{array}{cc}
0 & 1\\
-1 & 0
\end{array} \right)}_{e_2}, \underbrace{\left( \begin{array}{cc}
0 & \ii\\
\ii & 0
\end{array} \right)}_{e_3} \right\rangle,\]
with structure equations
\[[e_1, e_2] = 2e_3, \qquad [e_1, e_3] = -2e_2, \qquad [e_2, e_3] = 2e_1.\]
It is then easy to see that the only non vanishing coefficients $\mu^i_{jk}$ are
\[\mu^1_{23} = -2, \qquad \mu^2_{13} = 2, \qquad \mu^3_{12} = -2.\]
Hence the condition expressed in \eqref{eq: condition trivial canonical} is matched and so $\kod(X) = 0$.
\end{exam}

\begin{rem}
Recall that $\SU(2) \simeq S^3$, so there is another well known complex structure defined on $\SU(2) \times \SU(2)$ introduced by Calabi and Eckmann in \cite{Calabi-Eckmann}. Calabi--Eckmann manifolds are known to be complex manifolds of Kodaira dimension $-\infty$. Our example is interesting since it shows that it is possible to find non-integrable almost complex structures on a manifold whose Kodaira dimension is strictly greater than other integrable complex structures: this seems to be against the intuition that integrable almost complex structures have more chances to have sections than non-integrable ones.
\end{rem}

\begin{exam}[$G = \SO(4)$]
In this example, let $G = \SO(4)$. Its Lie algebra is
\[\mathfrak{so}(4) = \left\langle e_1, e_2, e_3, e_4, e_5, e_6 \right\rangle\]
with the following structure equations
\[\begin{array}{lllll}
[e_1, e_2] = -e_4, & [e_1, e_5] = e_3, & [e_2, e_4] = -e_1, & [e_3, e_4] = 0, & [e_4, e_5] = -e_6,\\
{[}e_1, e_3] = -e_5, & [e_1, e_6] = 0, & [e_2, e_5] = 0, & [e_3, e_5] = -e_1, & [e_4, e_6] = e_5,\\
{[}e_1, e_4] = e_2, & [e_2, e_3] = -e_6, & [e_2, e_6] = e_3, & [e_3, e_6] = -e_2, & [e_5, e_6] = -e_4.
\end{array}\]
Hence the condition expressed in \eqref{eq: condition trivial canonical} is matched and so $\kod(X) = 0$.
\end{exam}

\begin{exam}[$G = \Sp(2)$]
In this example, let $G = \Sp(2)$ be the compact symplectic group. Its Lie algebra is
\[\mathfrak{sp}(2) = \left\langle e_1, e_2, e_3, e_4, e_5, e_6, e_7, e_8, e_9, e_{10} \right\rangle,\]
where with respect to the standard quaternionic units $i$, $j$ and $k$
\[\begin{array}{llll}
e_1 = \left( \begin{array}{cc}
i & 0\\
0 & 0
\end{array} \right), & e_2 = \left( \begin{array}{cc}
j & 0\\
0 & 0
\end{array} \right), & e_3 = \left( \begin{array}{cc}
k & 0\\
0 & 0
\end{array} \right), & e_4 = \left( \begin{array}{cc}
0 & 1\\
-1 & 0
\end{array} \right),\\
e_5 = \left( \begin{array}{cc}
0 & i\\
i & 0
\end{array} \right), & e_6 = \left( \begin{array}{cc}
0 & j\\
j & 0
\end{array} \right), & e_7 = \left( \begin{array}{cc}
0 & k\\
k & 0
\end{array} \right), & e_8 = \left( \begin{array}{cc}
0 & 0\\
0 & i
\end{array} \right),\\
e_9 = \left( \begin{array}{cc}
0 & 0\\
0 & j
\end{array} \right), &  & e_{10} = \left( \begin{array}{cc}
0 & 0\\
0 & k
\end{array} \right). & 
\end{array}\]
An explicit computation of the commutators, hence of the coefficients $\mu^i_{jk}$ shows that the condition expressed in \eqref{eq: condition trivial canonical} is matched and so $\kod(X) = 0$.
\end{exam}

\begin{exam}[$G = \SU(3)$]
In this example, let $G = \SU(3)$. Its Lie algebra is
\[\mathfrak{su}(3) = \left\langle e_1, e_2, e_3, e_4, e_5, e_6, e_7, e_8 \right\rangle,\]
with
\[\begin{array}{ll}
e_1 = \left( \begin{array}{ccc}
\ii & 0 & 0\\
0 & 0 & 0\\
0 & 0 & -\ii
\end{array} \right), & e_2 = \left( \begin{array}{ccc}
0 & 0 & 0\\
0 & \ii & 0\\
0 & 0 & -\ii
\end{array} \right),\\
e_3 = \left( \begin{array}{ccc}
0 & 1 & 0\\
-1 & 0 & 0\\
0 & 0 & 0
\end{array} \right), & e_4 = \left( \begin{array}{ccc}
0 & \ii & 0\\
\ii & 0 & 0\\
0 & 0 & 0
\end{array} \right),\\
e_5 = \left( \begin{array}{ccc}
0 & 0 & 1\\
0 & 0 & 0\\
-1 & 0 & 0
\end{array} \right), & e_6 = \left( \begin{array}{ccc}
0 & 0 & \ii\\
0 & 0 & 0\\
\ii & 0 & 0
\end{array} \right),\\
e_7 = \left( \begin{array}{ccc}
0 & 0 & 0\\
0 & 0 & 1\\
0 & -1 & 0
\end{array} \right), & e_8 = \left( \begin{array}{ccc}
0 & 0 & 0\\
0 & 0 & \ii\\
0 & \ii & 0
\end{array} \right).
\end{array}\]
An explicit computation of the commutators, hence of the coefficients $\mu^i_{jk}$ shows that the condition expressed in \eqref{eq: condition trivial canonical} is matched and so $\kod(X) = 0$.
\end{exam}

To conclude, we present an explicit example where the group $G$ is not unimodular.

\begin{exam}
Let $G = \IR^4$ with the following group law:
\[(a_1, a_2, a_3, a_4) \star (x_1, x_2, x_3, x_4) = \pa{x_1 + a_1, e^{-a_1}x_2 + a_2, e^{-a_1}x_3 + a_3, x_4 + a_4}.\]
Then
\[E^1 = dx_1, \qquad E^2 = e^{x_1} dx_2, \qquad E^3 = e^{x_1} dx_3, \qquad E^4 = dx_4\]
defines a coframe of left invariant $1$-forms on $G$, with structure equations
\[dE^1 = 0, \qquad dE^2 = E^1 \wedge E^2, \qquad dE^3 = E^1 \wedge E^3, \qquad dE^4 = 0.\]
Then $G$ is not unimodular since, for example, $\operatorname{tr}\pa{[E_1, \blank]} = -2 \neq 0$. We consider a second copy of $G$ with coordinates $(y_1, y_2, y_3, y_4)$ and we define on $M = G \times G$ the invariant almost complex structure $J$ whose bundle of $(1, 0)$-forms is generated by
\[\varphi^i = v^i + \ii w^i, \qquad \text{where } v^i = \pi_i^* E^i, w^i = \pi_2^* E^i.\]
A direct computation with the structure equations shows that
\[\begin{array}{ll}
\delbar \varphi^1 = 0, & \delbar \varphi^2 = \frac{1}{4} (1 + \ii) \pa{\varphi^1 \wedge \bar{\varphi}^2 - \varphi^2 \wedge \bar{\varphi}^1},\\
\delbar \varphi^4 = 0, & \delbar \varphi^3 = \frac{1}{4} (1 + \ii) \pa{\varphi^1 \wedge \bar{\varphi}^3 - \varphi^3 \wedge \bar{\varphi}^1},
\end{array}\]
and so
\[\delbar\pa{\varphi^1 \wedge \varphi^2 \wedge \varphi^3 \wedge \varphi^4} = \underbrace{\frac{1}{2} (1 + \ii) \bar{\varphi}^1}_{\alpha} \wedge \varphi^1 \wedge \varphi^2 \wedge \varphi^3 \wedge \varphi^4.\]
Let $\cX_1, \cX_2, \cX_3, \cX_4$ be the vector fields dual to $\varphi^1, \varphi^2, \varphi^3, \varphi^4$ respectively. Then a smooth $m$-canonical section $f \cdot \pa{\varphi^1 \wedge \varphi^2 \wedge \varphi^3 \wedge \varphi^4}^{\tensor m}$ is pseudoholomorphic if and only if $f = u + \ii v$ satisfies \eqref{eq: pseudohol cond 3}, which reads as
\begin{equation}\label{eq: system non unimodular}
\left\{ \begin{array}{l}
\frac{\partial u}{\partial x_1} - \frac{\partial v}{\partial y_1} + mu - mv = 0\\
\frac{\partial u}{\partial y_1} + \frac{\partial v}{\partial x_1} + mu + mv = 0\\
e^{-x_1} \frac{\partial u}{\partial x_2} - e^{-y_1} \frac{\partial v}{\partial y_2} = 0\\
e^{-y_1} \frac{\partial u}{\partial y_2} + e^{-x_1} \frac{\partial v}{\partial x_2} = 0\\
e^{-x_1} \frac{\partial u}{\partial x_3} - e^{-y_1} \frac{\partial v}{\partial y_3} = 0\\
e^{-y_1} \frac{\partial u}{\partial y_3} + e^{-x_1} \frac{\partial v}{\partial x_3} = 0\\
\frac{\partial u}{\partial x_4} - \frac{\partial v}{\partial y_4} = 0\\
\frac{\partial u}{\partial y_4} + \frac{\partial v}{\partial x_4} = 0.
\end{array} \right.
\end{equation}
In order to solve this system, we set
\[u = e^{-mx_1 - mx_2} \cdot \hat{u}, \qquad v = e^{-mx_1 - mx_2} \cdot \hat{v}\]
and determine the conditions that $(\hat{u}, \hat{v})$ must satisfy. It turns out that they are (formally) the same as \eqref{eq: system non unimodular}, except the first two which become
\[\frac{\partial \hat{u}}{\partial x_1} - \frac{\partial \hat{v}}{\partial y_1} = 0, \qquad \frac{\partial \hat{u}}{\partial y_1} + \frac{\partial \hat{v}}{\partial x_1} = 0\]
respectively. Then we consider the following change of variables:
\[\begin{array}{rl}
(x_1, y_1, x_2, y_2, x_3, y_3, x_4, y_4) = & g(X_1, Y_1, X_2, Y_2, X_3, Y_3, X_4, Y_4) =\\
= & (X_1, Y_1, e^{-X_1} X_2, e^{-Y_1} Y_2, e^{-X_1} X_3, e^{-Y_1} Y_3, X_4, Y_4)
\end{array}\]
and we set
\[\hat{U} = \hat{u} \circ g, \qquad \hat{V} = \hat{v} \circ g.\]
The equations for $(\hat{U}, \hat{V})$ now are the classical Cauchy--Riemann equations, so the solutions are functions $\hat{F} = \hat{U} + \ii \hat{V}$ which are holomorphic in $Z_1 = X_1 + \ii Y_1$, $Z_2 = X_2 + \ii Y_2$, $Z_3 = X_3 + \ii Y_3$, $Z_4 = X_4 + \ii Y_4$. As a consequence, a smooth $m$-canonical section $f \cdot \pa{\varphi^1 \wedge \varphi^2 \wedge \varphi^3 \wedge \varphi^4}^{\tensor m}$ is pseudoholomorphic if and only if we can write $f$ as
\[e^{-m(x_1 + y_1)} \cdot \hat{F}\pa{x_1 + \ii y_1, e^{x_1}x_2 + \ii e^{y_1}y_2, e^{x_1}x_3 + \ii e^{y_1}y_3, x_4 + \ii y_4},\]
where $\hat{F}: \IC^4 \longrightarrow \IC$ is a holomorphic function.
\end{exam}

\section{Further geometrical properties of almost complex real parallelizable manifolds}

\subsection{Norden structures}
We recall the definition of Norden structure, introduced in \cite{Norden}.
\begin{defin} Let $(M,g)$ be an almost complex manifold with a pseudo-Riemannian metric $g$ such that $J$ is $g-$symmetric, then $(J,g)$ is called \emph {Norden structure} on $M$ and $(M,J,g)$ a \emph {Norden manifold}.
\end{defin} 

In our previous paper \cite {Cattaneo-Nannicini-Tomassini-2} we constructed natural Norden structures for almost complex $4-$dimensional solvmanifolds, analogous construction can be applied to real parallelizable smooth manifolds of even dimension.

Let $M$ be a real parallelizable smooth manifold of real dimension $2n$, $n\geq 2$; let $\{v_1,...,v_n,w_1,...,w_n\}$ be a basis of $\Gamma(M,TM)$ wich is a parallelism of $M$ and let $\{v^1,...,v^n,w^1,...,w^n\}$ be the dual frame.

We define the natural neutral pseudo-Riemannian metric $g$ on $M$ by:
$$g=-v^1\otimes v^1-v^2\otimes v^2-...-v^n\otimes v^n+w^1 \otimes w^1+w^2 \otimes w^2+...+w^n \otimes w^n.$$
Let $J$ be the almost complex structure on $M$ defined in the standard form:
$$Jv_i=w_i, \qquad Jw_i=-v_i$$
for all $i=1,...,n$.

Direct computation gives the following.
\begin{prop} $(J,g)$ is a Norden structure on $M$. Moreover the twin metric $\tilde g$, defined by $\tilde g(X,Y)=g(JX,Y)$ for all $X,Y \in \Gamma(M,TM)$, is given by:
$$\tilde g=v^1\otimes w^1+v^2\otimes w^2+...+v^n\otimes w^n+w^1 \otimes v^1+w^2 \otimes v^2+...+w^n \otimes v^n.$$
\end{prop}

\subsection{Quasi-statistical structures}
Statistical structures, introduced in \cite{Lauritzen}, are pairs $(g,\nabla)$ of a pseudo-Riemannian metric $g$ and a torsion-free affine connection on a smooth manifold $M$ such that $\nabla g$ is totally symmetric. Statistical manifolds admitting torsion, or quasi-statistical manifolds, were introduced in \cite{K}, we recall here the definition.

\begin{defin} Let $(M,g)$ be a pseudo-Riemannian manifold and let $\nabla$ be an affine connection on $M$ with torsion tensor $T^{\nabla}$.  $(g,\nabla)$ is called a \emph{quasi-statistical structure} on $M$, and $(M,g,\nabla)$ a \emph{quasi-statistical manifold}, or \emph{statistical manifold admitting torsion}, if
$$(\nabla_X g)(Y)-(\nabla _Y g)(X)+g(T^{\nabla} (X,Y))=0$$
for any $X,Y \in \Gamma(M,TM)$.
\end{defin}
A general construction of quasi-statistical structures by means of a pseudo-Riemannian metric, an affine connection and a tensor field on a smooth manifold, is given in \cite{Blaga-Nannicini}. As a particular case of this construction, we have the following.
\begin{prop} (\cite  {Blaga-Nannicini}) Let $(M,J,g)$ be a Norden manifold, let $\nabla^g$ be the Levi-Civita connection of $g$ and let $\eta$ be a non zero $1-$form on $M$. Then $(g, \bar \nabla:=\nabla ^g + J\otimes \eta)$ is a quasi-statistical structure on $M$.
\end{prop}
\begin{cor} For a real parallelizable manifold $M$, of even dimension $2n$, the Norden structure $(J,g)$ defined before and the $1-$forms $\{v^1,...,v^n,w^1,...,w^n\}$, give  $2n$ quasi-statistical structures on $M$, namely: $\{(g, {\bar \nabla}^i:=\nabla ^g + J\otimes v^i)\}_{1\leq i\leq n}$, $\{(g,{\bar \nabla}^{n+i}:=\nabla ^g + J\otimes w^i)\}_{1\leq i\leq n}$.
\end{cor}
We recall  the notion of dual connections which is strictly related to those of statistical and quasi-statistical structures, \cite{Amari}, \cite{Lauritzen}.
\begin{defin} Let $(M,g)$ be a pseudo-Riemannian manifold, two affine connections $\nabla$ and $\nabla^*$ on $M$ are said to be \emph{dual connections} with respect to $g$ if 
$$X(g(Y,Z))=g(\nabla_X Y,Z)+g(Y,\nabla^*_X Z)$$ 
for any $X,Y,Z \in \Gamma(TM)$ and $(g,\nabla,\nabla^*)$ is called a \emph{dualistic structure}.
\end{defin}
Direct computation gives the following.
\begin{lemma} If $\nabla$ is an affine connection on $(M,g)$ then the dual connection $\nabla^*$ of $\nabla$ with respect to $g$ is uniquely defined by:
$${\nabla^*}_X Y=\nabla_X Y+g^{-1}((\nabla_X g)(Y))$$
for any $X,Y\in \Gamma(TM)$. 
\end{lemma}

\subsection{Almost K\"ahler structures on the tangent bundle}
Let $(M,g)$ be a pseudo-Riemannian manifold and let $\nabla$ be an affine connection on $M$. Let $\pi:TM \rightarrow M$ be the canonical projection and let $\pi_*:T(TM) \rightarrow TM$ be the tangent map of $\pi$. In \cite {Dombrowski}, § 2, Dombrowski constructed a map  $K:T(TM) \rightarrow TM$, depending on the connection $\nabla$, with the property that for every $a\in TM$ and for every $X, Y\in T_{\pi_*(a)}M$ there exists exactly one $Z\in T_a(TM)$ such that $\pi_*(Z)=X$ and $KZ=Y$, (cf. \cite {Dombrowski}, Lemma 1). We can use these maps to decompose $T(TM)$ in horizontal and vertical subbundles with respect to $\nabla$, namely $T(TM)=H(T(TM))\oplus V(T(TM))$ with $V(T(TM)) = \ker \pi_*$ and $H(T(TM)) = \ker K$. This decomposition allows us to define a natural almost complex structure $\hat J={\hat J}^\nabla$ on $TM$ as follows: for $Z \in T_a(TM)$ define $\hat J (Z)\in T_a(TM)$ as the unique vector such that 
$$\pi_*(\hat J Z)=-KZ, \, \, K(\hat J Z)=\pi_*(Z).$$ 
Moreover if $Z=Z^h+Z^v$ is the decomposition of $Z$ in horizontal and vertical components we have $\hat J (Z^h)\in V(T(TM))$ and $\hat J (Z^v)\in H(T(TM))$, \cite{Dombrowski}.\\
It is known that $\hat J$ is integrable if and only if $\nabla$ is flat and torsion free, \cite{Dombrowski}.\\
Let $\hat G={{\hat G}^\nabla}_g$ be the metric on $TM$ defined by $(g,\nabla)$ as $\hat G:=g\oplus g.$ This metric was introduced by Sasaki, in the fundamental paper \cite{Sasaki}, for tangent bundles of Riemannian manifolds and then generalized for tangent bundles of pseudo-Riemannian manifolds, \cite{Dida}. It is now a standard notion in differential geometry, called the {\it{Sasaki metric}}.\\
We can easily verify that $(TM,\hat J, \hat G)$ is an almost Hermitian manifold with K\"ahler form $\omega(\, ,\,):=\hat G(\hat J\, ,\,)$.

The following is known, \cite{Satoh}. 
\begin{prop} For $(TM,\hat J, \hat G)$ induced by $(M,g,\nabla)$ the following statement are equivalent:\\
1. $d\omega=0$;\\
2. $T^{{\nabla}^*}=0$;\\
3. $(M,g,\nabla)$ is a quasi-statistical manifold.
\end{prop}

As a consequence of our construction and of previous proposition, we get the following.

\begin{prop} Let $M$ be a real parallelizable manifold of real even dimension $2n$, let $(J,g)$ be the Norden structure on $M$ defined before and let $(g, {\bar \nabla}^\alpha)_{1\leq \alpha \leq 2n}$ be the corresponding family of quasi statistical structures. Then $\mathbb W=(TM, {{\hat {G}}_g}^{\nabla^\alpha}, {{\hat {J}}}^{\nabla^\alpha})$ is an almost K\"ahler manifold.
\end{prop}

\vspace{1cm}
\noindent
{\bf Conflict of interest.}
On behalf of all authors, the corresponding author states that there is no conflict of interest.

\bibliographystyle{alpha}

\begin{thebibliography}{99}

\bibitem{Abbena-Grassi}
E. Abbena and A. Grassi, {\it Hermitian left invariant metrics on complex Lie groups and cosymplectic Hermitian manifolds}, Boll. Un. Mat. Ital. A (6) {\bf 5} (1986), no. 3.

\bibitem{Amari}
S.I. Amari, {\it Differential-Geometrical Methods in Statistics}, Lecture Notes in Statistics, Springer, {\bf 28} (1985).

\bibitem{BCHM}
C. Birkar, P. Cascini, C.D. Hacon and J. M$^\text{c}$Kernan. {\it Existence of minimal models for varieties of log general type}. J. Amer. Math. Soc., {\bf 23} (2010), no. 2, 405--468.

\bibitem{Blaga-Nannicini}
A.M. Blaga and A. Nannicini, {\it On Statistical and Semi-Weyl Manifolds Admitting Torsion}, Mathematics, {\bf 10} (2022).

\bibitem{Calabi-Eckmann}
E. Calabi and B. Eckmann, {\it A class of compact, complex manifolds which are not algebraic}, Annals of Mathematics, Second Series, {\bf 58} (1953).

\bibitem{Cattaneo-Nannicini-Tomassini-1}
A. Cattaneo, A. Nannicini, and A. Tomassini. {\it Kodaira dimension of almost K\"ahler manifolds and curvature of the canonical connection}, Ann. Mat. Pura Appl. (4) {\bf 199} (2020), no. 5, 1815--1842.

\bibitem{Cattaneo-Nannicini-Tomassini-2}
A. Cattaneo, A. Nannicini and A. Tomassini, {\it On Kodaira dimension of almost complex $4$-dimensional solvmanifolds without complex structures}, Internat. J. Math., {\bf 32}(10) (2021).

\bibitem{Cirici-Wilson}
J. Cirici and S.O. Wilson, {\it Dolbeault cohomology for almost complex manifolds}, Adv. Math., {\bf 391} (2021).

\bibitem{Chen-Zhang-1}
H. Chen and W. Zhang, {\it Kodaira dimensions of almost complex manifolds I}, arXiv preprint arXiv:1808.00885v2 [math.DG] (2018). Accepted by American Journal of Mathematics.

\bibitem{Dida}
H. M. Dida, A. Ikemakhen, {\it A class of metrics on tangent bundles n the geometry of the tangent bundle of pseudo-Riemannian manifolds}, Archivium Mathmaticum (Brno), {\bf 47} (2011), no. 4, 293-308.

\bibitem{Dombrowski}
P. Dombrowski, {\it On the geometry of the tangent bundle}, J. Reine Angew. Math., {\bf 210} (1962), 73-88.

\bibitem{K}
T. Kurose, {\it Statistical Manifolds Admitting Torsion}. Geometry and Something; Fukuoka Univ.: Fukuoka-shi, Japan, 2007. (In Japanese)

\bibitem{Lauritzen}
S.L. Lauritzen, {\it Statistical manifolds}, Differential Geometry in Statistical Inferences. IMS Lecture Notes Monogr. Ser. {\bf{10}} (1987), 1163-216.

\bibitem{Milnor}
J. Milnor, {\it Curvatures of left invariant metrics on Lie groups}, Advances in Math., {\bf 21}(3) (1976), 293-329.

\bibitem{Kawamata}
Y. Kawamata, Yujiro. {\it Finite generation of a canonical ring}. Current developments in mathematics, 2007, 43--76, Int. Press, Somerville, MA, 2009.

\bibitem{Norden}
A.P. Norden, {\it On a class of four-dimensional $A$-spaces}, Izv. Vys\v{s}. U\v{c}ebn. Zaved. Matematika, {\bf 4}(17) (1960), 145-153.

\bibitem{Sasaki}
S. Sasaki, {\it On the differential geometry of the tangent bundles of Riemannian manifolds}, Tohoku Math. J., {\bf 10} (1958), 338-354.

\bibitem{Satoh}
H. Satoh, {\it Almost Hermitian structures on tangent bundles}, Proceeding of the Eleventh International Workshop on Differential Geometry, {\bf 11} (2007).

\bibitem{Wolf1}
J.A. Wolf, {\it On the geometry and classification of absolute parallelisms I}, J. Differential Geometry, {\bf 6} (1971/72).

\end{thebibliography}

\end{document}